\documentclass[11pt]{amsart}
\usepackage{amssymb,latexsym,amscd, amsthm, epsfig}

\newtheorem{Thm}{Theorem}[section]
\newtheorem{Cor}[Thm]{Corollary}
\newtheorem{Prop}[Thm]{Proposition}
\newtheorem{Lem}[Thm]{Lemma}
\theoremstyle{definition}
\newtheorem{Def}[Thm]{Definition}
\newtheorem{Ex}[Thm]{Example}
\newtheorem{Rmk}[Thm]{Remark}

\newtheorem{Conj}[Thm]{Conjecture}

\date{}

\newcommand{\OO}{{\mathcal{O}}}
\newcommand{\PP}{{\mathbb{P}}}
\newcommand{\CC}{{\mathbb{C}}}

\begin{document}

\title{On the identifiability of binary Segre products}

   \author{Cristiano Bocci}
   \address{Cristiano Bocci\\
Universit\'a degli Studi di Siena\\
     Dipartimento di Scienze Matematiche e Informatiche\\
     Pian dei Mantellini, 44\\
     I -- 53100 Siena
     }
   \email{bocci24@unisi.it}
     \urladdr{http://www.mat.unisi.it/newsito/docente.php?id=16}

\author{Luca Chiantini}
      \address{Luca Chiantini\\
Universit\'a degli Studi di Siena\\
     Dipartimento di Scienze Matematiche e Informatiche\\
     Pian dei Mantellini, 44\\
     I -- 53100 Siena
     }
   \email{chiantini@unisi.it}
  \urladdr{http://www.mat.unisi.it/newsito/docente.php?id=4}
   
\maketitle {\footnotesize \emph{\textbf{Abstract. We prove that a product of $m>5$ copies of $\PP^1$, embedded in the projective
space $\PP^r$ by the standard Segre embedding, is $k$-identifiable (i.e.
a general point of the secant variety $S^k(X)$ is contained in
only one $(k+1)$-secant  $k$-space), for
all $k$ such that $k+1\leq 2^{m-1}/m$.
} \\ \\
\indent \textbf{MSC. 14M12, 14N05, 62H17}}
\\ ${}$
\maketitle

\section{Introduction}

In this paper, we study secant varieties $S^k$ of Segre products
of projective spaces, with special focus on products of many copies
of $\PP^1$ (binary Segre products or Bernoulli models, in Algebraic Statistics). 
We are mainly concerned with the number of secant spaces passing through
a  general point of a secant variety.

In the literature, one finds several methods for computing the
dimension of secant varieties of products. Let us just mention  the
inductive method introduced by Abo, Ottaviani and Peterson 
in \cite{AOP}, which provides a procedure for detecting
when the dimension coincides with the expected one.
In the specific case of products of copies of $\PP^1$, a complete description
of the dimension of secant varieties has been obtained by Catalisano, Geramita
and Gimigliano in \cite{CGG1}
and \cite{CGG2}. When the number of copies $m$ of $\PP^1$ is bigger than $4$, they
prove that $S^k$ has always the expected dimension.
From our point of view, the result  implies
that, when the secant variety $S^k$ does not fill the ambient space and $m>4$, then
through a general point of $S^k$ one finds only finitely many $(k+1)$-secant 
$k$-spaces. In this paper, we go one step further and we ask  {\it how many} secant
spaces one finds through a general point of $S^k$.  Our main result is:

\begin{Thm}\label{main} 
Let $X$ be a product of $m>5$ copies of $\PP^1$, embedded in the projective
space $\PP^r$, $r=2^m-1$, by the standard Segre embedding.
Let $S^k(X)$ be the $k$-th secant variety of $X$, generated by $(k+1)$-secant $k$-spaces. If $k+1\leq 2^{m-1}/m$,
 then a general point of $S^k(X)$ is contained in
only one $(k+1)$-secant  $k$-space.\end{Thm}

Following a notation suggested by applications to Algebraic Statistics,
we say that a variety $X$ is {\it $k$-identifiable} when through a
{\it general} point of the secant variety $S^k(X)$, there is only one
$(k+1)$-secant $k$-space. (Those who would prefer "generically $k$-identifiable" here,
should consider that there are {\it always} points of $S^k(X)$, e.g. points
of $X$, for which the number jumps to infinity.)  
With this notation, our result can be rephrased by saying that a product
of $m>5$ copies of $\PP^1$ is $k$-identifiable, as soon as $k+1\leq 2^{m-1}/m$ (i.e. 
$m - log_2(m) \geq \lceil \log_2 (k+1) \rceil +1$).
\smallskip

 From this last point of view, $k$-identifiability has been studied
because of its application to Algebraic Statistics and other fields.
Using methods of Algebraic Geometry, Elmore, Hall and Neeman proved in
\cite{EHN} the following asymptotic result: when the number $m$ of copies of
$\PP^1$ is ``very large" with respect to $k$, then the binary Segre product is
$k$-identifiable. 

As far as we know, the best bound  for identifiability of binary products
has been obtained by Allman, Matias and Rhodes 
in \cite{AMR} (Corollary 5). They prove that the product is $k$-identifiable when
$m> 2\lceil\log_2 (k+1)\rceil +1$. Thus, they
give a lower bound for $2^m$ which is quadratic with respect to $k+1$. Our theorem provides an extension of these results. In order to compare with the aforementioned bounds, notice that 
$(\PP^1)^m$ cannot be $k$-identifiable for $k>2^m/(m+1) -1$, by a simple
dimensional count, explained in Section \ref{background}.
Thus, the maximal $k$ for which identifiability makes sense is
$k_{max}=\lfloor 2^m/(m+1) \rfloor -1$. The result of Allman, Matias and Rhodes,
rewritten from this point of view, says that $(\PP^1)^m$ is 
$k$-identifiable for $k+1\leq 2^{(m-1)/2}$.  Our theorem extends this bound, for we prove that:
$$ (\PP^1)^m \mbox{ is  $k$-identifiable for } k+1\leq \frac{2^{m-1}}{m}.$$  

Still more or less half way from the maximum,
but a sensible improvement, anyway. For example, for  $m=10$, $k_{max}$ is $92$, our bound 
proves the $k$-identifiability for $k\leq 50$, while  Allman, 
Matias and Rhodes give identifiability for $k\leq 16\sqrt 2$.
\smallskip

Our method is strongly based on the result on the dimension
of secant varieties, contained in \cite{CGG2}. Indeed, for a variety $X$, both the
dimension of secant varieties and the number of secant spaces 
passing through a point, are linked to the existence of very 
degenerate subvarieties passing through $k+1$ general points
of $X$. We explain this fact in details, through the next section.
Using this remark, transferring results on the dimension
of secant varieties to results on the identifiability,
becomes straightforward. At the end of the paper, we will 
explain why we need the assumption $m>5$.
Namely, we prove that $(\PP^1)^5$ is not $4$-identifiable.

Let us finish by stating the following conjecture, suggested by
our analysis.

\begin{Conj}\label{conge} For $m>5$ and for all $k=1,\dots,
\lfloor 2^m/(m+1) \rfloor -1$, the binary Segre product
$(\PP^1)^m$ is $k$-identifiable.
\end{Conj}

\section{Geometric background}\label{background}

In this section, we collect some known results on secant varieties and Segre products.
We refer to \cite{CC2}, for details and proofs. We work over the complex field 
and we consider the projective space $\PP^r = \PP^r_{\CC}$,
equipped with the tautological line bundle $\OO_{\PP^r}(1)$.

If $Y$ is a subset of $\PP^r$, we denote by $\langle Y \rangle$
the linear span of $Y$. We say that $Y$ is {\it non--degenerate} 
if $\langle Y \rangle=\PP^r$.
A linear subspace of dimension $n$ of $\PP^ r$
will be called a \emph{$n$--subspace} of $\PP^ r$.

Let $X\subset \PP^ r$ be an irreducible, projective, non--degenerate variety of
dimension $m$. For any non--negative integer $k$, 
the {\it $k$--secant variety} of $X$ is the Zariski closure in $\PP^r$ of the union of all
$k$--dimensional subspaces of $\PP^r$ that are spanned by $k+1$
independent points of $X$. We denote it by $S^k(X)$, or $S^k$, if
no confusion arises. $S^k(X)$ can be seen as the closure of the image, 
under the second projection,
of the {\it abstract secant variety}, i.e. the incidence subvariety 
$AbS^k(X)\subset X^{(k)}\times \PP^r$,
$$AbS^k(X) =\{((P_0,\dots,P_k),P): P\in\langle P_0,\dots,P_k\rangle, \mbox{ and the
$P_i$'s are independent}\}.$$
Notice that $AbS^k(X)$ is {\it always} a variety of dimension $mk+m+k$. 
When $X\subset \PP^ r$ is reducible, the same definition of secant variety
holds, except that we only consider linear spaces meeting every component
of $X$. In particular, when $X$ has $k+1$ components, the secant
variety coincides with the {\it join} of the components
(see \cite{Zak}).

\begin{Def} We say that $X$ {\it has $k$-th secant order $\mu$} 
if for a general point $P\in S^k(X)$, there are exactly $\mu$ unordered $(k+1)$-uples
$P_0,\dots,P_k$ of points of $X$ such that $P\in \langle P_0,\dots,P_k\rangle$.
We say that $X$ is (generically) {\it $k$-identifiable}
if it has $k$-th secant order $1$, i.e.
if for a general point $P\in S^k(X)$, there is a unique unordered $(k+1)$-uple
$P_0,\dots,P_k$ of points of $X$ such that $P\in \langle P_0,\dots,P_k\rangle$.
\end{Def}

\begin{Ex}\label{linsp} If $X$ is the union of $k+1$ linearly independent subspaces
of dimension $m$, then $X$ has $k$-th secant order $1$. This is an easy exercise of Linear Algebra, for $k+1=2$. For $k+1>2$,
it follows by induction, by projecting from one linear component of $X$. Rational normal curves in $\PP^{2k+1}$ are the unique irreducible curves with
$k$-th secant order $1$. See e.g. Theorem 3.1 of \cite{CC2}.
\end{Ex}

From the definition of secant varieties, it follows that:
\begin{equation} \label{defect} s^{(k)}(X):= \dim (S^k(X))\leq
 \min\{r,mk+m+k\}.\end{equation}
The right hand side is called the {\it expected dimension} of $S^{k}(X)$.

\begin{Def}
We say that $X$ is {\it $k$--defective} when a strict inequality
holds in (\ref{defect}). 
\end{Def}

\begin{Rmk}
It is clear that $X$ is $k$-identifiable when the projection $AbS^k(X)\to S^k(X)$
is birational. So $X$ cannot be $k$-identifiable when $\dim(AbS^k(X))> s^{(k)}(X)$. In particular, $X$ is not $k$-identifiable when $r<mk+m+k$ or when $X$ is 
$k$-defective.
\end{Rmk}

Let $X\subset \PP^ r$ be a variety. We denote by ${\rm Sing}(X)$ the Zariski-closed
subset of singular points of $X$.
Let $P\in X\setminus{\rm Sing}(X)$ be a smooth point. We
denote by $T_{X,P}$ the embedded tangent space to $X$ at $P$,
which is a $m$--subspace of $\PP^ r$. More generally, if $P_0,\ldots,P_k$
are smooth points of $X$, we will set

$$T_{X,P_0,\ldots,P_k}=\langle \bigcup_{i=1}^ n T_{X,P_i} \rangle.$$
The relations between secant varieties and tangent spaces to $X$
are enlightened by the celebrated Terracini's Lemma: 

\begin{Lem} (see \cite{Terr1} or, for modern versions, \cite{Adl}, 
\cite{Dale1}, \cite{Zak}).
Given a general point $P\in S^k(X)$, lying in the subspace  $\langle P_0,\dots,P_k
\rangle$ spanned by $k+1$ general points on $X$, then the tangent space 
$T_{S^{k}(X),P}$ to $S^{k}(X)$ at $P$ is the span $T_{X,P_0,\dots,P_k}$ 
of the tangent spaces $T_{X,P_0},\dots, T_{X,P_k}$.
\end{Lem}

Using the correspondence between the abstract secant variety and
$S^k(X)$, one obtains from Terracini's Lemma, a condition for the defectivity
of $X$:

\begin{Thm}\label{def-weakdef} (See \cite{CC2}, Theorem 2.5)
Let $P_0,\dots ,P_k$ be general points of $X$.
If $H$ is a general hyperplane tangent to $X$ at $P_0,\dots ,P_k$, we can consider 
the {\it contact variety} of $H$, i.e. the union $\Sigma$ of the irreducible components of 
${\rm Sing}(X\cap H)$.  If $X$ is $k$-defective, then $\Sigma$ is positive dimensional. 
\end{Thm}

\noindent The previous Theorem suggests a refinement of the notion 
of defective variety.

\begin{Def} An irreducible, non--degenerate variety $X\subset\PP^r$
such that $s^{(k)}(X)<r$ is {\it $k$--weakly defective} if for
$P_0,...,P_k\in X$ general points, the general hyperplane $H$
containing $T_{X,P_0,...,P_k}$ is tangent to $X$ along a variety
$\Sigma(H)$ of positive dimension. $\Sigma(H)$ is called the {\it
$(k+1)$--contact variety} of $H$.\
\end{Def}

\noindent It turns out that $k$-defective implies $k$--weakly defective,
but the converse is false. We refer to \cite{CC0} and \cite{CC2}
for a discussion on the subject.

The main link between identifiability and weakly defective varieties 
lies in the following:

\begin{Thm}\label{mu1} (See \cite{CC2}, Corollary 2.7)
Let $X\subset\PP^r$ be an irreducible,
projective, non--degenerate variety of dimension $m$. Assume $
mk+m+k<r$. Then $X$ is $k$-identifiable, unless it is
$k$--weakly defective. \end{Thm} 

\begin{Thm}\label{mu2} (See \cite{CC2}, Theorem 2.4)
Let $X\subset\PP^r$ be an irreducible,
projective, non--degenerate variety of dimension $m$. Assume $
mk+m+k<r$ and assume that $X$ is $k$-weakly defective. 
Call $\Sigma$ a general $(k+1)$-th contact variety.
Then, the $k$-th secant order of $\Sigma$ is equal to the $k$-th
secant order of $X$. 
\end{Thm} 

\noindent Thus, a way to prove that a variety $X$ is $k$-identifiable, at least
when $r\neq mk+m+k$, is to prove that $X$ is not $k$-weakly defective,
or, if it is $k$-weakly defective, that the general contact variety 
$\Sigma$ has $k$-th secant order $1$.
\smallskip

The second cornerstone in our theory links $k$-defectivity
and $k$-weakly defectivity with
the existence of degenerate subvarieties, passing through $k+1$ general
points in $X$. Namely, if $X$ is $k$-defective or 
$k$-weakly defective, then it turns out
that the general contact variety  is highly degenerate. 

\begin{Thm}\label{wdef} (See \cite{CC2}, Theorem 2.4 and Theorem 2.5)
Assume $mk+m+k<r$. 
If $X$ is $k$-weakly defective, then a general contact variety $\Sigma$
spans a linear space of dimension $\leq nk+n+k$, where $n=\dim(\Sigma)$.
Moreover, $X$ is $k$-defective if and only if 
$\Sigma$ spans a space of dimension $< nk+n+k$.
\end{Thm} 

In conclusion, we obtain:

\begin{Cor} \label{basic}
Assume $r>mk+m+k$. Assume that for all $n=1,\dots,m-1$, there
are no families of $n$-dimensional subvarieties of $X$, whose general element 
spans a linear space of dimension $\leq nk+n+k$ and passes through
$k+1$ general points of $X$. Then $X$ is not $k$-weakly defective.
Hence it is $k$-identifiable.
\end{Cor}

This is our starting point. In the next sections, we will obtain the
$k$-identifiability of products of $\PP^1$'s, for $k$ in our range, 
by proving that subvarieties of Segre products
$\PP^1\times\dots\times\PP^1$ passing through $k+1$ general points cannot
be too degenerate, unless they are formed by a bunch of 
independent linear spaces. One should observe that both Corollary \ref{basic} and 
the second part of Theorem \ref{mu1} cannot be inverted. 

\begin{Ex} The existence of families of degenerate subvarieties
cannot guarantee that $X$ is $k$-weakly defective. For instance, 
consider $X=\PP^2$ embedded in $\PP^9$ by the $3$-Veronese
embedding. Then $X$ is not $1$-weakly defective. Indeed,
the general hyperplane tangent to $X$ at two general points cuts
a divisor which corresponds, in $\PP^2$, to
a general cubic curve with two singular points. Such a cubic splits in
the union of a conic and a line, and it is reduced. On the other hand, 
through $2$ general points of $X$ one finds a curve
spanning a space of dimension $1\cdot 1+1+1=3$. Namely, it is
 the twisted cubic, image of the line trough the two points.
\end{Ex}

\begin{Ex} When $X$ is $k$-weakly defective,  it can
be $k$-identifiable as well. This may happen, by  \cite{CC2}, Theorem 2.4,
when the contact locus has $k$-th secant order $1$. 
Examples of such varieties can be found in \cite{CC2}, Example 3.7, 
but they are singular. A smooth example was communicated us by G. Ottaviani. 
Take the Segre embedding of $X=\PP^1\times\PP^1\times\PP^2$ in $\PP^{11}$.
Using a computer-aided procedure, one can find that the general
hyperplane which is tangent to $X$ at two points, is indeed tangent
along a twisted cubic. The computation was indeed performed at
two specific points of $X$, but notice that Aut($X$) acts
transitively on pair of points. Thus $X$ is $1$-weakly defective.
Since a twisted cubic curve has first secant order equal to $1$
(Example \ref{linsp}),
it turns out by \cite{CC2}, Theorem 2.4, that $X$ is $1$-identifiable.
The $1$-identifiability of $X$ also follows from the Kruskal's
identifiability criterion for the product of three projective spaces
(see \cite{K}). 
\end{Ex}

As a consequence, one cannot use the inverse of the previous argument
to determine  the non-identifiability
of a variety $X$, simply by studying degenerate subvarieties.

\begin{Rmk} \label{monod}
The degenerate subvarieties $\Sigma$, whose existence is guaranteed by
Corollary \ref{basic} in any weakly defective variety, are not necessarily 
smooth, neither they are necessarily irreducible 
(although one can assume that they are reduced).
On the other hand, one can use a monodromy--type argument 
(see e.g. \cite{CC1}, Proposition 3.1) in order to show that,
when $\Sigma$ is reducible, all the components are interchanged in a flat
deformation, thus they are general members of a flat family. 
This is due to the generality of the points $P_i$'s. 
In particular, we may assume that the components share 
the same geometrical properties, also with respect to the linear series 
induced by the projections to the factors $\PP^1$.
\end{Rmk}

\section{Proof of the Theorem}

Let us start with an useful Lemma of Linear Algebra:

\begin{Lem}\label{linalg} Let $H_0,\dots,H_k$ be subspaces of
$\PP^s$, such that the sum $H_0+\dots+H_k$ is not direct.
Let $p$ be the dimension of the linear span of the $H_i$'s.
Then, for a general choice of points $P_0,\dots,P_k\in \PP^r$
$r\geq 1$, the dimension of the linear span of the spaces 
$H_i\times \{P_i\}\subset \PP^s\times\PP^r$ is at least $p+1$.
\end{Lem}
\begin{proof} We may assume that $H_0$ meets the span
$L=\langle H_1\cup\dots\cup H_k\rangle$. If $d=\dim(H_0)$
and $e=\dim(L)$, then by assumption $\langle H_0\cup L\rangle$
has dimension at most $d+e$, while for a general choice of the points
$P_0$, $P_1$,
$$\dim\langle (H_0\times P_0)\cup (L\times P_1)\rangle =d+e+1,$$
for the two spaces belong to linearly independent copies 
of $\PP^s$ in the product.
Now, the claim follows by specializing all $P_2,\dots, P_k$ to $P_1$.
\end{proof}

The proof of our Main Theorem now follows soon by the main result of \cite{CGG2}
and by the following general observation.

\begin{Lem} Let $Y\subset \PP^s$ be a non-degenerate variety of dimension $d$
which is not  $k$-defective ($k\geq 1$). Assume $kd+k+d<s$.
Then $X=Y\times \PP^q$ ($q\geq 1$) is $k$-identifiable. 
\end{Lem}
\begin{proof}
If $X$ is $k$-weakly defective, then by Theorem \ref{wdef}
 the $(k+1)$-contact locus is a 
subvariety $W$ of some dimension $n>0$, contained in $\PP^{nk+n+k}$,
which passes through $k+1$ general points of $X$. 
Assume that such a variety exists. Call $W'\subset Y$ the image of $W$
in the projection $X\to Y$ and call $n'=\dim(W')$.
Since $W'$ passes through $k+1$ general points of $Y$, and $Y$
is not $k$-defective, then the span of $W'$ has dimension at least
$n'k+n'+k$, by the second part of Theorem \ref{wdef}.
Now, notice that the fibers of the projection $W\to W'$ span a space of
dimension at least $n-n'$ in $\PP^q$. It follows, by Linear Algebra,
that $W$ spans a space of dimension at least $(n-n'+1)(n'k+n'+k)$.

Now, we have to study several cases.
Assume $0<n'<n$. Then $n'(n-n')\geq n-n'$, so that $(n-n'+1)n'\geq n$.
Moreover $(n-n'+1)k>k$. It follows that 
$(n-n'+1)(n'k+n'+k)$ is bigger than $nk+n+k$, so we get a
contradiction.

Assume $n=n'>0$. By construction, the linear series $L$ which sends $W$ to
$\PP^s$, passing through the embedding $W\subset X$ and the projection
$X\to Y$, has dimension equal to the span of $W'$. Hence it has
dimension $kn+n+k$, in our case.  Call $L'$ the linear series defining
the projection $W\to \PP^q$. If the image of $W$ in $\PP^q$ has dimension
at least $1$, then the embedding of $W$ in $\PP^s\times\PP^q$ is given
by a series $L+L'$, whose dimension is at least $\dim(L)+1=nk+n+k+1$.
It follows that $W$ spans a space of dimension at least $nk+n+k+1$,
a contradiction. Since $W$ passes through $k+1\geq 2$ points of $X$,
its image into $\PP^q$ can be trivial only when $W$ is given by 
$k+1$ components, $W=W_0\cup\dots\cup W_k$ and  each $W_i$ is contained 
in a fiber of  the projection $X\to \PP^q$.
By monodromy (Remark \ref{monod}), each $W_i$
has the same dimension $n$ and spans a space of the same dimension 
$q'$, in the fiber. Call $H_0,\dots, H_k$ the projections of these
spaces to $\PP^s$.  Since $W'$ spans a space of dimension
$nk+n+k$, then the $H_i$'s span a space of the same dimension.
Thus, if the $H_i$'s are not linearly independent, then 
$W$ spans a space of dimension at least $nk+n+k+1$, by Lemma
\ref{linalg}, a contradiction.
It follows that the span of the $H_i$'s has dimension 
$nk+n+k=q'k+q'+k$, so that $q'=n$. This means that 
each $W_i$ projects to a subspace $H_i$ of dimension $n$
in $\PP^s$. Since each $W_i$ sits in a fiber of $X\to \PP^q$,
this implies that each $W_i$ is linear,
and these subspaces are independent.  

Assume $n'=0$. Then necessarily $W$ consists of $k+1$ components
$W=W_0\cup\dots\cup W_k$, each $W_i$ being contained in a fiber of 
the projection. As above, it turns out that $W$ spans a space of dimension 
$kq'+q'+k$. Thus $q'>n$ yields a contradiction. Hence $q'=n$, so that every
$W_i$ is linear.

It follows, from the previous analysis, that necessarily  
$W$ is the union of $k+1$ linearly 
independent subspaces of dimension $n$.
We get then that either $X$ is not $k$-defective, so it is $k$-identifiable by
Theorem \ref{basic}, or it is $k$-defective, and the $(k+1)$-contact variety
$W$ is the union of  $k+1$ linearly independent subspaces of dimension $n$.
In the latter case, the $k$-th secant order of $W$ is known to be $1$
(see Example \ref{linsp}).
Thus $X$ is $k$-identifiable, by Theorem \ref{mu2}.
\end{proof}

{\bf Proof of the main Theorem}
Let $k$ the any positive integer such that  $(k+1)m<2^{m-1}-1$, 
so that the $k-$secant variety of $(\PP^1)^{m-1}$ cannot span $\PP^{2^{m-1}-1}$.
By the main result of \cite{CGG2}, $(\PP^1)^{m-1}$ is not $k-$defective.
Then the previous Lemma implies that $(\PP^1)^m$ is $k$-identifiable.
\hfill\qed
\vskip0.3cm
Indeed, the previous Lemma also prove the following, stronger
statement:

\begin{Thm} 
Let $X$ be a product of $m>5$ copies of $\PP^1$, embedded in the projective
space $\PP^r$, $r=2^m-1$, by the standard Segre embedding.
If $k+1\leq 2^{m-1}/m$, then $X$ is not $k$-weakly defective.
\end{Thm}
\begin{proof} We know from the proof of the previous Lemma, that
if $X$ is $k$-weakly defective, then a general hyperplane $H$ tangent at 
$k+1$ general points  $P_0,\dots, P_k$ of $X$, is 
tangent along a union of linear spaces.
Thus it can only be tangent along fibers of some projection $X\to\PP^1$,
because the product does not contain other lines.
This is clear for $\PP^1\times\PP^1$, while for higher
dimensional products one can argue by induction, on
some projection $(\PP^1)^m\to(\PP^1)^{m-1}$.
Thus, by symmetry, $X$ can be $k$-weakly defective
only  when a general $(k+1)$-tangent hyperplane $H$ is tangent along
all the fibers passing through  $P_0,\dots, P_k$.
But then, for a general choice of a point $Q$ in some
fiber passing through $P_0$, a general hyperplane tangent   
to $P_0,P_1,\dots, P_k$ is also tangent at $Q,P_1,\dots,P_k$.
By the same argument, it is also tangent along any fiber 
passing through $Q$.
Arguing again in this way, we get that a general $H$ must
be tangent (thus must contain) any point of $X$. 
An obvious contradiction.
\end{proof}

\section{Results for small $m$}

\begin{Prop} The product $X$ of $5$ copies of $\PP^1$ is not $4$-identifiable.
Through a general point of $S^5(X)$ one finds exactly two $5$-secant, $4$-spaces.
\end{Prop}
\begin{proof} Indeed, we prove that through $5$ general points of
$X$ one can find an irreducible elliptic normal curve $C\subset \PP^9$,
contained in $X$. Since a general point of the $\PP^9$, spanned 
by $C$, sits in exactly two subspaces of dimension $4$, $5$-secant to
an irreducible elliptic normal curve
(by \cite{CC2} Proposition 5.2), it follows that the $4$-th secant order
 of $X$ is at least $2$. In particular, $X$ is $4$-weakly defective,
 by \cite{CC2}, proposition 2.7, and the $4$-th contact locus contains an elliptic
 normal curve as $C$. A computer aided computation, at 
 $5$ specific  points of $X$,  proves that indeed the $5$-contact locus of
 $X$ is exactly an irreducible elliptic normal curve of degree $12$. 
 The computation has been performed with the Macaulay2 Computer
 Algebra package \cite{GS}, with the script described in \cite{BC}.
 Thus  $4$-th secant order  of $X$ is $2$ 
 (by Theorem \ref{mu2}) and the claim is proved.

To prove the existence of the curve $C$ passing through $5$ 
general points $P_0,\dots,P_4$ of $X$, we start with the product of three lines $X'=\PP^1
\times\PP^1\times\PP^1$. Through the $5$ points $P'_0,\dots,P'_4\in X'$,
projection of the $P_i$'s, one can find a $2$-dimensional 
family $\mathcal F$ of elliptic normal curves 
$C'$ of degree $6$.  Indeed $X'\subset\PP^5$ is a sestic threefold with elliptic
curve sections, and there is a $2$-dimensional family of hyperplanes passing
through $5$ general points of $X$. $\mathcal F$ is parametrized by
points of some plane $\Pi$, obtained by projecting $\PP^5$ from
the span of the $P'_i$'s.

Consider now the product $X''$ of the two remaining copies of $\PP^1$,
so that $X=X'\times X''$. We also get $5$ distinguished general points
$P''_0,\dots, P''_4\in X''$. For any curve $C'$ of the family $\mathcal F$, we have
a $7$-dimensional family of embeddings $C'\to X''$. Thus, adding
the automorphisms of $C'$, for $C'\in\mathcal F$
general, we may assume that each $P'_i$, $i=1,\dots,4$,
goes to the corresponding $P''_i$. The condition that $P'_0$ goes to $P''_0$ 
determines two algebraic condition on the family, hence two algebraic curves on 
$\Pi$. Thus, there is at least one curve $C'$ of the family, for which
$P'_0$ goes to $P''_0$. This determines an elliptic normal curve $C$ in $X$,
passing through the $5$ given general points $P_i$'s.
The fact that  $C$ is irreducible, for a general choice
of the points, follows by the computer-aided computation, on a specific example. 
\end{proof}

For $6$ copies of $\PP^1$, the maximal value for which $k$-identifiability
makes sense is $k_{max}=9$.
Our result gives that the product $X$ of $6$ copies of 
$\PP^1$ is  $k$-identifiable, for $k=1,\dots, 5$.
The $k$-identifiability of $(\PP^1)^6$, for $k=6,7,8,9$, can be directly
checked by a computer-aided procedure.
Indeed, the following observation reduces our problem to check
only what happens for the maximal number $k$
such that $mk+m+k<r$.

\begin{Prop}\label{maxk}
Assume that $km+m+k<r$ and $X$ is not $k$-weakly defective. Then
$X$ is not $(k-1)$--weakly defective.
\end{Prop}
\begin{proof}
Fix $k+1$ general points $P_0,\dots,P_k\in X$. The family of hyperplanes
containing the tangent space $T_{X,P_1,\dots,P_k}$ is irreducible,
so a general hyperplane tangent to $X$ at $P_0,\dots,P_k$ is the limit
of a family of hyperplanes tangent at $P_1,\dots,P_k$. Since the general
element of this last family has a zero dimensional contact locus,
the claim follows.
\end{proof}

Now, by Corollary \ref{basic}, it is enough to compute that
some hyperplane tangent to $X$ at some points $P_0,\dots, P_k$,
is in fact tangent only at those $k+1$ points.
Using this procedure with $9$ points of $(\PP^1)^6$, a computer-aided
compution, using the script in \cite{BC},
proves the following:

\begin{Prop} \label{sei} For $m=6$ and for all $k\leq k_{max}= 9$, 
the product $X$ of $6$ copies of $\PP^1$ is  $k$-identifiable. 
\end{Prop}

For sure, with a more advanced technical equipment, one can 
analyze products with  more copies of $\PP^1$.
Nevertheless, Proposition \ref{sei} already provides an initial 
evidence for our Conjecture \ref{conge}.

\end{document}